\documentclass[10pt]{article}
\usepackage{amssymb}
\usepackage{amsthm}
\usepackage{amsmath}
\usepackage{amscd}
\usepackage{stmaryrd}
\usepackage{array}
\usepackage{url}
\usepackage{hyperref}
\usepackage[latin2]{inputenc}
\usepackage{t1enc}
\usepackage{enumerate}
\usepackage[mathscr]{eucal}
\usepackage[british]{babel}
\usepackage[dvips]{graphicx}
\usepackage[dvips]{epsfig}
\usepackage{epsf}
\usepackage{color}
\usepackage{indentfirst}
\usepackage[all]{xy}

\newcommand{\Ker}{\mathrm{Ker}}

\newcommand{\val}{\mathrm{val}}

\newcommand{\ord}{\mathrm{ord}}

\newcommand{\Gal}{\mathrm{Gal}}

\newtheorem{thm}{Theorem}
\newtheorem{pro}[thm]{Proposition}
\newtheorem{lem}[thm]{Lemma}
\newtheorem{cor}[thm]{Corollary}

\theoremstyle{definition}
\newtheorem*{rem}{Remark}
\newtheorem*{rems}{Remarks}

\def\Z{\mathbb{Z}}
\def\Q{\mathbb{Q}}

\title{The distribution and density of cyclic groups of the reductions of an elliptic curve over a function field}
\date{}
\author{M\'arton Erd\'elyi \\ merdelyi@freestart.hu}

\begin{document}
\maketitle

\begin{abstract}
Let $K$ be a global field of finite characteristic $p\geq2$, and let $E/K$ be a non-isotrivial elliptic curve. We give an asympotoic formula of the number of places $\nu$ for which the reduction of $E$ at $\nu$ is a cyclic group. Moreover we determine when the Dirichlet density of those places is 0.
\end{abstract}

\section{Statement of results}

Let $K$ be a global field of characteristic $p$ and genus $g_K$, and let $k=\mathbb{F}_q \subset K$ ($q=p^f$) be the algebraic closure of $\mathbb{F}_p$ in $K$. We denote by $V_K$ the set of places of $K$. For $\nu \in V_K$, we denote by $k_\nu$ the residue field of $K$ at $\nu$, and by $\deg(\nu):= [k_\nu : \mathbb{F}_q]$ the degree of $\nu$. Let $\overline{k}$ be an algebraic closure of $k$. Denote $\phi:(x\mapsto x^q)\in\Gal(\overline{k}/k)$ the $q$-Frobenius. Let $k_r|k$ be the unique degree $r$ extension in $\overline{k}$.

Let $E/K$ be an elliptic curve over $K$ with $j$-invariant $j_E \notin k$, which we shall standardly call non-isotrivial. We denote by $V_{E/K}$ the set of places of $K$ for which the reduction $E_\nu/k_\nu$ is smooth and $|\overline{V}_{E/K}|=\sum_{\nu\notin V_{E/K}}\mathrm{deg}(\nu)$. For $n\in\mathbb{N}\setminus\{0\}$ let $V_{E/K}(n)=\{\nu\in V_{E/K}|\deg(\nu)=n\}$.

From the theory of elliptic curves we know that for $\nu\in V_{E/K}$, $E_\nu(k_\nu) \simeq \Z/d_\nu \Z \times \Z/d_\nu e_\nu \Z$ for nonzero integers $d_\nu, e_\nu$, uniquely determined by $E$ and $\nu$. We call the integers $d_\nu$ and $d_\nu e_\nu$ the elementary divisors of $E_\nu$.

The goal of this paper is to extend the results of \cite{CT} about the distribution of the places $\nu \in V_{E/K}$ for which $E_\nu(k_\nu)$ is a cyclic group. Such questions have been investigated for the reductions of an elliptic curve defined over $\Q$ (e.g. in \cite{BaSh}, \cite{Co1}, \cite{Co2}, \cite{CoMu}, \cite{GuMu}, \cite{Mu1}, \cite{Mu2}, \cite{Se2}), mainly in relation with the elliptic curve analogue of Artin's primitive root conjecture formulated by Lang and Trotter in \cite{LaTr}. This latter conjecture was investigated in the function field setting $E/K$ by Clark and Kuwata \cite{ClKu}, and by Hall and Voloch \cite{HaVo} (see also Voloch's work on constant curves \cite{Vo1}, \cite{Vo2}). In \cite{ClKu}, a particular emphasis was placed on the study of the cyclicity of $E_\nu(k_\nu)$.

~\\

In this paper we obtain an explicit asymptotic formula for the number of places $\nu \in V_{E/K}$, of fixed degree, for which $E_\nu(k_\nu)$ is cyclic. Our result is a direct extension of the work of \cite{CT} which worked in finite characteristic $p>3$.

\begin{thm}\label{aszimptotika}
Let $E/K$ be a non-isotrivial elliptic curve. For all $\varepsilon>0$ there exists $c=c(K,E,\varepsilon)$ such that for all $n\in\mathbb{N}$ we have
\[\left|\#\left(\nu\in V_{E/K}(n)|E_\nu(k_\nu)\mathrm{~is~cyclic}\right)-\delta(E/K,1,n)\frac{q^n}{n}\right|\leq c\frac{q^{n/2+\varepsilon}}{n},\]
where
\[\delta(E/K,1,n)=\sum_{\stackrel{m\leq q^{n/2}+1}{m|q^{n-1}}}\frac{\mu(m)\mathrm{ord}_m(q)}{|K(E[m]):K|},\]
a $\mu$ is the Moebius function and $\ord_m(q)$ denotes the multiplicative order of $q$ modulo $m$ for $m\in\mathbb{N}$, $(m,q)=1$.
\end{thm}

Here the second parameter of $\delta$ refers to $d_\nu=1$. The exact same calculation for \break $\#\left(\nu\in V_{E/K}(n)|d_\nu=d\right)$ yields a same result.

We are also able to answer a previously unsolved question concerning the Dirichlet density $\delta(E/K,1)$ of places $\nu$ such that $E_\nu(k_\nu)$ is cyclic: we can characterize when this density is 0.

\begin{thm}\label{suruseg}
Let $E/K$ be a non-isotrivial elliptic curve. Then $\delta(E/K,1)=0$ if and only if $\delta(E/K,1,1)=0$.
\end{thm}

Surprisingly this can happen in the case, when the torsion subgroup of $E(K)$ is cyclic, as well.

~\\

We sketch the original proof of Theorem \ref{aszimptotika} in the following:

With simple inclusion-exclusion principle we get
\[\#\left(\nu\in V_{E/K}(n)|E_\nu(k_\nu)\mathrm{~is~cyclic}\right)=\sum_m\mu(m)\#\left(\nu\in V_{E/K}(n)|(\mathbb{Z}/m\mathbb{Z})^2\leq E_\nu(k_\nu)\right),\]
moreover the sum has very few nonzero terms: if $(\mathbb{Z}/m\mathbb{Z})^2\leq E_\nu(k_\nu)$ then
\begin{itemize}
\item by Hasse's theorem $|E_\nu(k_\nu)|\leq q^n+1+2\sqrt{q^n}$, thus $m\leq q^{n/2}+1$,
\item by the Weil-pairing the cyclotomic field $\mathbb{F}_q(\zeta_m)\leq k_\nu$, thus $m|q^n-1$.
\end{itemize}

By \cite{CT} Corollary 10, $(\mathbb{Z}/m\mathbb{Z})^2\subseteq E_\nu(k_\nu)$ if and only if $\nu$ splits completely in $K(E[m])/K)$ or equivalently the conjugacy class of the Frobenius at $\nu$ in $\Gal(K(E[m])/K)\leq\mathrm{GL}_2(\mathbb{Z}/m\mathbb{Z})$ is the set consisting of the identity element. Thus let $c_m$ be the integer for which the algebraic closure of $k$ in $K(E[m])$ is $\mathbb{F}_{q^{c_m}}$ and $\pi_1(n,K(E[m]/K)=\#(\nu\in V_{E/K}(n)|\nu\mathrm{~splits~completely~in~}K(E[m])/K))$.

Note that $c_m=\mathrm{ord}_m(q)$, which corresponds to the algebraic part of the field extension $K(E[m])/K$ and $\mathrm{Gal}(K(E[m])/K\mathbb{F}_{q^{c_m}})\leq\mathrm{SL}_2(\mathbb{Z}/m\mathbb{Z})$ describes the geometric part.

This enables us to use an effective version of the Chebotarev density theorem (\cite{MuSc} Theorem 2): we obtain that if $m,n\in\mathbb{N}$ such that $(m,p)=1$ and $\mathrm{ord}_m(q)|n$, then there exists $\rho=\rho(E,K,m)$ such that
\[\left|\pi_1(n,K(E[m])/K)-\frac{\mathrm{ord}_m(q)\cdot q^n}{[K(E[m]):K]}\right|\leq2\left((3g_K+(\rho+1)|\overline{V}_{E/K}|)\frac{q^{n/2}}{n}+\frac{|\overline{V}_{E/K}|}{2n}\right)+|\overline{V}_{E/K}|.\]

If $K(E[m])/K$ is at most tamely ramified for all $n$ (consequently if $p>3$) then $\rho=0$. The contribution of the paper to the proof is that we handle wildly ramified field extensions to bound $\rho$ independently from $m$. For this we need to do some local computation, this is contained in Section 2.

Now we can simply sum up these estimations and by standard arguments prove the theorem.

~\\

In Section 3 we investigate the density $\delta(E/K,1)$.

If $K(E[\ell])=K$ for some prime $\ell\neq p$ or equivalently if the torsion subgroup of $E(K)$ is not cyclic, it is clear that for all $\nu\in V(E/K)$ the group of $E_\nu(k_\nu)$ is not cyclic either. It is a natural question to ask whether the converse is true. In $\cite{CT}$ is proven that for the special case $K=\mathbb{F}_q(j_E)$ the answer is affirmative.

We start by showing an elliptic curve $E/K$ with cyclic torsion subgroup such that for infinitely many $n\in\mathbb{N}$ the value of $\delta(E/K,1,n)$ is 0 (and for at least one $n$ we have \break $\#(\nu\in V_{E/K}(n)|E_\nu(k_\nu)$ is cyclic$)=0$.)

Then we prove Theorem \ref{suruseg} and finally we construct an elliptic curve $E/K$ with cyclic torsion subgroup for which $\delta(E/K,1,1)=0$ and hence $\delta(E/K,1)$ is also 0.

\section{Chebotarev density theorem for wildly ramified extensions}

Let $L|K$ be a Galois extension of function fields with constant field $k$, and unramified away from a set of places $S$. Let $|S|=\sum_{\nu\in S}\deg(\nu)$ and let $c$ be the integer such that the algebraic closure of $k$ in $L$ is a degree $c$ extension of $k$. Let $\pi_1(n,L/K)$ be the number of places of degree $n$ of $K$ which split completely in $L/K$. $\rho_{L/K}$ is an integer as defined in \cite{Se3} 1.2 and \cite{MuSc} 3 and which we will redefine and estimate thereafter. 

We will use the following version of Chebotarev density theorem for global function fields:

\begin{thm} (\cite{MuSc}, Theorem 2.)
If $c|n$, then
\[\left|\pi_1(n,L/K)-c\frac{1}{[L:K]}|V_K(n)|\right|\leq2\left((3g_K+(\rho_{L/K}+1)|S|)\frac{q^{n/2}}{n}+\frac{|S|}{2n}\right)+|S|.\]
Otherwise $\pi_1(n,L/K)=0$.
\end{thm}

Recall the definition of $\rho_{L/K}$:

By the abuse of notation let first $K$ denote a local field, with ring of integers $o_K$, maximal ideal $m_K\lhd o_K$ and standard valuation $\mathrm{val}_K:K\to\mathbb{Z}\cup\{\infty\}$.

Let $L|K$ be a finite, totally ramified extension with ramification index $e_{L/K}$, different ideal $\mathcal{D}_{L/K}\lhd o_L$ and let us denote by $\val_{L}(\mathcal{D}_{L/K})$ the expontent of $m_L$ in $\mathcal{D}_{L/K}$: the integer $n$ for which $\mathcal{D}_{L/K}=m_L^n$.

We then have $p\nmid e_{L/K}\iff \val_L(\mathcal{D}_{L/K})=e_{L/K}-1\iff L|K$ is at most tamely ramified. (\cite{Na}, Theorem 4.8)

Thus there exists an integer $0\leq j<e_{L/K}$ such that $\val_L(\mathcal{D}_{L/K})\equiv -j-1~(\mathrm{mod~} e_{L/K})$. Then $j=0 \iff L|K$ is at most tamely ramified. Finally let

\[\rho_{L/K}=\left\{\begin{array}{ll}0, & \mathrm{~if~} L/K \mathrm{~is~at~most~tamely~ramified,} \\ \frac{\val_L(\mathcal{D}_{L/K})-(e_{L/K}-j-1)}{e_{L/K}}, & \mathrm{~if~} L/K \mathrm{~is~wildly~ramified,}\end{array}\right.\]

Notice that this is really an integer and $\rho_{L/K}=\left\lceil\frac{\val_L(\mathcal{D}_{L/K})+1}{e_{L/K}}\right\rceil-1$, where $\lceil\cdot\rceil$ is the ceiling function. Remark that if we would have Hensel's bound $\val_L(D_{L/K})\leq e_{L/K}-1+\val_p(e_{L/K})$ (which fails in the function field case), then we would get $\rho_{L/K}=1$ if $L|K$ is wildly ramified and 0 otherwise.

Also note that if $L=K(\alpha)$ with a separable Eisenstein polynomial $f$ satisfying $f(\alpha)=0$, then $e_{L/K}=\deg(f)$ and $\val_L(\mathcal{D}_{L/K})=\val_L(f'(\alpha))$.

Now, return to the original situation, that is, $K$ denotes a global function field, and $L/K$ is a Galois extension with constant field $k$. Then let $\rho_{L/K}=\mathrm{max}_\nu(\rho_{L_\nu/K_\nu})$.

\begin{lem}\label{rhobecsles}
Let $M|L|K$ be a tower of totally ramified Galois extensions with constant field $k$. We then have
\begin{enumerate}
\item $\rho_{L/K}\leq \rho_{M/K}\leq \rho_{L/K}+\left\lceil\frac{\rho_{M/L}}{e_{L/K}}\right\rceil$.
\item $\rho_{M/K}=\rho_{L/K}$ if $M|L$ is at most tamely ramified.
\end{enumerate}
\end{lem}

\begin{proof}
Clearly it suffices to verify the statements locally. Let us once again denote by $K=K_\nu$, $L=L_\nu$ and $M=M_\nu$. Then
\begin{eqnarray*}
\rho_{M/K}=\left\lceil\frac{\val_M(\mathcal{D}_{M/K})+1}{e_{M/K}}\right\rceil-1=\left\lceil\frac{\val_M(\mathcal{D}_{M/L})+\val_M(\mathcal{D}_{L/K})+1}{e_{M/L}\cdot e_{L/K}}\right\rceil-1= \\
=\left\lceil\frac{1}{e_{L/K}}\cdot\frac{\val_M(\mathcal{D}_{M/L})+1}{e_{M/L}}+\frac{\val_L(\mathcal{D}_{L/K})+1}{e_{L/K}}-\frac{1}{e_{L/K}}\right\rceil-1.\end{eqnarray*}

Using the trivial inequality $\lceil a+b\rceil\leq\lceil a\rceil+\lceil b\rceil$ we get the upper bound in 1. The lower bound is also clear since $(\val_M(\mathcal{D}_{M/L})+1)/e_{M/L}\geq 1$, hence 

\[\frac{1}{e_{L/K}}\left(\frac{\val_M(\mathcal{D}_{M/L})+1}{e_{M/L}}-1\right)\geq 0.\]

Moreover equality holds if and only if $M|L$ is at most tamely ramified, which proves 2.
\end{proof}

\begin{pro}\label{Chebotarev}
Let $E/K$ be a non-isotrivial elliptic curve over $K$ and $m,n\in\mathbb{N}$ such that $(m,p)=1$ and $\mathrm{ord}_m(q)|n$. Then there exists $\rho$ independent from $n$ such that
\[\left|\pi_1(n,K(E[m])/K)-\frac{\mathrm{ord}_m(q)\cdot q^n}{[K(E[m]):K]}\right|\leq2\left((3g_K+(\rho+1)|\overline{V}_{E/K}|)\frac{q^{n/2}}{n}+\frac{|\overline{V}_{E/K}|}{2n}\right)+|\overline{V}_{E/K}|.\]
\end{pro}

\begin{proof}
We use the fact that there exits a finite extension $K'|K$ such that $E/K'$ has either good or split multiplicative reduction over $K'$ (\cite{Si1} Proposition VII.5.4.), hence for all $m$ we have that $K(E[m])/K$ is at most tamely ramified (\cite{Si2} Theorem 10.2). Then by Lemma \ref{rhobecsles} we get that $\rho_{K(E[m])/K}\leq\rho_{K'(E[m])/K}=\rho_{K'/K}=:\rho$, which does not depend on $m$.
\end{proof}

Theorem \ref{aszimptotika} follows from this by standard arguments.

\begin{rems}
\begin{enumerate}
\item Let $k=\mathbb{F}_2$, $K=k(j)$ and $E/K$ be the elliptic curve with $j$-invariant $j$ defined by the equation $y^2+xy+x^3+j^{-1}=0$. Here $K(E)/K$ is wildly ramified at $\infty$. However if we consider the Deuring normal form $E': y^2+txy+y+x^3=0$, some computation with Tate's algorithm show that there is no more wild ramification. The $j$-invariant of $E'$ is $t^{12}/(t^3+1)$. Hence we can set $L=K(t)=k(t)$ with $t$ satisfying $f(t)=t^{12}+jt^3+j=0$ - doing that we adjoin the coordinates of 2 points of the 3-torsion. Here we have $e_{L/K}=12$ and \break $\val_L(\mathcal{D}_{L/K})=\val_L(f'(t))=14$ since $f$ is an Eisenstein polynomial. Thus \break $\rho_{L/K}=\left\lceil\frac{\val_L(\mathcal{D}_{L/K})+1}{e_{L/K}}\right\rceil-1=1$.

We need one more field extension, since $E$ and $E'$ are not isomorphic over $L$, only over $M=L(s)$ with $s^2+ts+t=0$. Here we have $e_{M/L}=2$ and $\val_M(\mathcal{D}_{M/L})=4$. Hence as in the proof of \ref{rhobecsles}
\[\rho_{M/K}=\left\lceil\frac{\val_M(\mathcal{D}_{M/L})+e_{M/L}\cdot\val_L(\mathcal{D}_{L/K})+1}{e_{M/L}\cdot e_{L/K}}\right\rceil-1=\left\lceil\frac{33}{24}\right\rceil-1=1.\]

Hence in this case we have
\[\left|\pi_1(n,K(E[m])/K)-\frac{2^n}{|\mathrm{SL}_2(\mathbb{Z}_m)|}\right|\leq8\cdot\frac{2^{n/2}+1}{n}+2.\]

For $a\in\mathbb{F}_{2^n}^*$ let $E(a)/\mathbb{F}_{2^n}:y^2+xy+x^3+a=0$ and let $E(0)/\mathbb{F}_{2^n}:y^2+y+x^3=0$. Denote $f(n)=\#(a\in\mathbb{F}_{2^n}|E(a)\mathrm{~is~cyclic})=\sum_{d|n}d\#(\nu\in V_{E/K}(d)|d_\nu=1)$. Here only the term $d=n$ is relevant, thus $f(n)\simeq g(n)+O(2^{n/2})$, where $g(n)=2^n\sum_{\stackrel{m\leq 2^{n/2}+1}{m|2^n-1}}\frac{1}{|\mathrm{SL}_2(\mathbb{Z}/m\mathbb{Z})|}$. We computed some values of $f(n)$ and $g(n)$, and the following tables illustrate the result.
\[\begin{array}{ccc}
\begin{array}{|c|c|c|c|}
\hline
n&f(n)&g(n)&f(n)-g(n)\\
\hline
1&2&2&0\\
\hline
2&3&3.83&-0.83\\
\hline
3&8&8&0\\
\hline
4&15&15.2&-0.2\\
\hline
5&32&32&0\\
\hline
6&60&61.14&-1.14\\
\hline
7&128&128&0\\
\hline
8&246&243.22&2.78\\
\hline\end{array} & ~ &
\begin{array}{|c|c|c|c|}
\hline
n&f(n)&g(n)&f(n)-g(n)\\
\hline
9&512&510.48&1.52\\
\hline
10&977&980.55&-3.55\\
\hline
11&2047&2047.83&-0.83\\
\hline
12&3873&3878.98&-5.98\\
\hline
13&8192&8192&0\\
\hline
14&15670&15701.13&-31.13\\
\hline
15&32673&32669.37&3.63\\
\hline
16&62294&62265.91&28.09\\
\hline\end{array}
\end{array}\]

Note that if $2^n-1>3$ is a Mersenne prime, then our esimate is sharp: there is no nontrivial $m$ such that $m|2^n-1$, thus $\delta(E/K,1,n)=1$ and also $E_\nu(k_\nu)$ is cyclic for all $\nu\in V_{E/K}(n)$.

\item Let $k=\mathbb{F}_3$ and $K=k(j)$ and $E:y^2+xy-x^3+j^{-1}=0$ the elliptic curve with $j$-invariant $j$ over $K$. We have wild ramification at $\infty$ again.

Now set $L=K(\mu)$ with $\mu^{12}-j((\mu^4-1))^2=0$. $E/L$ is isomorphic with \break $E': y^2=x(x-1)(x-\mu^2+1)$ - this is a bit varied version of the Legendre normal form composed with a quadratic extension. Here is no more wild ramification and again we have $\rho_{L/K}=1$.

The Chebotarev density theorem in this case gives:
\[\left|\pi_1(n,K(E[m])/K)-\frac{3^n}{|\mathrm{SL}_2(\mathbb{Z}_m)|}\right|\leq8\cdot\frac{3^{n/2}+1}{n}+2.\]
\end{enumerate}
\end{rems}

\section{Dirichlet density of places with cyclic reduction}

Let $n\in\mathbb{N}\setminus\{0\}$. Recall that
\[\delta(E/K,1,n)=\sum_{\stackrel{m\geq 1}{m|q^n-1}}\frac{\mu(m)\mathrm{ord}_m(q)}{|K(E[m]):K|}.\]

It is mentioned in \cite{CT} Remark 17, that if there exists a prime $\ell\neq p$ such that $K(E[\ell])=K$ (or equivalently the torsion subgroup of $E(K)$ is not cyclic), then for all $n$ we have $\delta(E/K,1,n)=0$. Moreover if $K=\mathbb{F}_q(j_E)$, then the converse holds.

We show that it is not true in general:

\begin{pro}
There exists an elliptic curve $E/K$ with cyclic torsion subgroup such that for infinitely many $n\in\mathbb{N}$ we have $\delta(E/K,1,n)=0$.
\end{pro}

\begin{proof}
Let $K=\mathbb{F}_5(t)$. We construct a curve $E$ such that the extension $K(E[3])|K$ is algebraic and of degree 2. Then since $\mathrm{ord}_3(5)=2$, and if $2|n$ we have
\[\delta(E/K,1,n)=\sum_{\stackrel{m\geq 1}{m|5^n-1}}\frac{\mu(m)\mathrm{ord}_m(5)}{|K(E[m]):K|}=\sum_{\stackrel{3\nmid m}{1\leq m|5^n-1}}\left(\frac{\mu(m)\mathrm{ord}_m(5)}{|K(E[m]):K|}-\frac{\mu(m)\mathrm{ord}_{3m}(5)}{|K(E[3m]):K|}\right).\]
Then either we have $2|\mathrm{ord}_m(5)=\mathrm{ord}_{3m}(5)$ and $K(E[3m])=K(E[3])K(E[m])=K(E[m])$, since $K(E[m])$ contains the cyclotomic field $K(\zeta_m)\geq K(\zeta_3)=K(E[3])$. Or $2\nmid\mathrm{ord}_m(5)$ and hence $\mathrm{ord}_{3m}(5)=2\cdot\mathrm{ord}_m(5)$, moreover $K(E[3])\nleq K(E[m])$ since $(\mathrm{ord}_m(q),\mathrm{ord}_3(q))=1$, consequently $|K(E[3m]):K(E[m])|=2$. Thus all terms on the right-hand side are 0, and $\delta(E/K,1,n)=0$.

To realize an explicit example set $p(t)=t^3-t^2+2t$, $q(t)=2t^6+t^5+t-1$ and \break $E: y^2=x^3+p(t)x+q(t)$ over $K$. Then $\Delta_E=-16(4p(t)^3+27q(t)^2)\neq 0$ and $j_E=1728\cdot\frac{4p(t)^3}{\Delta_E}\notin\mathbb{F}_q$, thus $E$ is non-isotrivial.

Moreover $E_0:y^2=x^3-1$ has 6 points over $\mathbb{F}_5$, hence the torsion subgroup of $E(K)$ is cyclic.

The third division polynomial is 
\[\psi_3(x)=3(x^4+2p(t)x^2-q(t)x-2p(t)^2)=3(x^2-(1+2t^2)x-(t^4-2t^3-t-1))(x+2t^2)(x+1),\]
where on the right-hand side the first term is irreducible over $\mathbb{F}_5[t]$. However if we \break denote $L=K(\sqrt{2})$, the following points are in $E(L)[3]$: $(-1,\pm\sqrt{2}(t^3-t^2+2t-2))$, \break $(-2t^2,\pm2(t^3-2t^2+2t+1))$, hence $E(L)[3]$ is the whole 3-torsion. Thus $K(E[3])=L$ and $|K(E[3]):K|=2$, indeed.
\end{proof}

\begin{rems}
\begin{enumerate}
\item For $n=2$ it is easy to verify, that there is no place $\nu\in V_{E/K}(2)$ such that $E_\nu$'s group is cyclic. Thus $\#(\nu\in V_{E/K}(n)|d_\nu=1)=0\not\Longrightarrow \exists \ell\neq p: K(E[\ell])=K$.
\item The same can be carried out for $q\neq 5$, $q\equiv2~(\mathrm{mod}~3)$. For example if $q=2$ we can choose $K=\mathbb{F}_2(t)(u)$, where $u^2+(t^3+1)u+(t^{12}+1)=0$ and $E/K:y^2+xy=x^3+(t^{12}+t^9+t^6+t^3)$. If $q\not\equiv2~(\mathrm{mod}~3)$, then we shall use a different prime $\ell$ instead of 3 such that $\mathrm{ord}_\ell(q)>1$.
\end{enumerate}
\end{rems}

~\\

Now we shall turn to a slightly different question. Whether the same phenomenon can arise if we consider all places $\nu\in V_{E/K}$ at once. Recall that by the definition of Dirichlet density we have
\[\delta(E/K,1)=\lim_{s\to1+0}\frac{\sum_{\stackrel{\nu\in V_{E/K}}{d_\nu=1}}q^{-s\deg(\nu)}}{\sum_{\nu\in V_{E/K}}q^{-s\deg(\nu)}}.\]

Of course, if the torsion subgroup of $E(K)$ is not cyclic, then by definition $\delta(E/K,1)=0$. Our goal is to determine when $\delta(E/K,1)=0$ in general.

Recall that for all but finitely many primes $\ell$ we have $\mathrm{Gal}(K(E[\ell])/K\mathbb{F}_{q^{\mathrm{ord}_\ell(q)}})\simeq\mathrm{SL}_2(\mathbb{Z}_\ell)$
(\cite{Ig} Theorem 4, \cite{Se1}, \cite{CT} Theorem 6) Let $M(E/K)$ be the torsion conductor of $E/K$ - the product of the finitely many exceptional primes $\ell_i$ and $N(E/K)$ be the least common multiple of $\mathrm{ord}_{\ell_i}(q)$. Moreover if $m_1,m_2\in\mathbb{N}$ such that $(m_1,p)=(m_2,p)=(m_2,M)=1$ and $m_1$ is composed of primes dividing $M(E/K)$, then $K(E[m_1])\cap K(E[m_2])=K\mathbb{F}_{q^{(\mathrm{ord}_{m_1}(q),\mathrm{ord}_{m_2}(q))}}$. (\cite{CT} Corollary 8)

Now we are ready to prove Theorem \ref{suruseg}:

\begin{proof}
First assume that $\delta(E/K,1,1)>0$. Let $N=N(E/K)$ and $M=M(E/K)$. If $(n,N)=1$, then in the definition of $\delta(E/K,1,n)$ all $m|q^n-1$ can be written in the form $m=m_0m'$ with $(m_0,M)=1$ and $m'|q-1$. Note that $\mathrm{ord}_m(q)=\mathrm{ord}_{m_0}(q)$. Using the previously mentioned facts we get $|K(E[m]):K|=|K(E[m_0]):K|\cdot|K(E[m']):K|$ and we can proceed as in \cite{CT} Remark 17:
\begin{eqnarray*}
\delta(E/K,1,n)=\sum_{\stackrel{m_0|q^n-1}{(m_0,q-1)=1}}\left(\sum_{m'|q-1}\frac{\mu(m_0m')\mathrm{ord}_{m_0m'}(q)}{|K(E[m_0m']):K|}\right)=\\
=\sum_{\stackrel{m_0|q^n-1}{(m_0,q-1)=1}}\frac{\mu(m_0)}{|K(E[m_0]):K\mathbb{F}_q^{\mathrm{ord}_{m_0}(q)}|}\left(\sum_{m'|q-1}\frac{\mu(m')}{|K(E[m']):K|}\right)=\\
=\left(\sum_{\stackrel{m_0|q^n-1}{(m_0,q-1)=1}}\frac{\mu(m_0)}{|\mathrm{SL}_2(\mathbb{Z}/m_0\mathbb{Z})|}\right)\cdot\delta(E/K,1,1)>\delta(E/K,1,1)\prod_{\ell\nmid pM}\left(1-\frac{1}{\ell(\ell^2-1)}\right)=\varepsilon>0.
\end{eqnarray*}

Now from the definition of $\delta(E/K,1)$ we have
\begin{eqnarray*}
\delta(E/K,1)=\lim_{s\to1+0}\frac{\sum_{n>n_0}\sum_{\stackrel{\nu\in V_{E/K}(n)}{d_\nu=1}}q^{-sn}}{\sum_{n>n_0}|V_{E/K}(n)|q^{-sn}}\geq\frac{\sum_{\stackrel{n\equiv1~(\mathrm{mod}~N)}{n>n_0}}\#(\nu\in V_{E/K}(n)|d_\nu=1)q^{-sn}}{\sum_{\stackrel{n\equiv1~(\mathrm{mod}~N)}{n>n_0}}N|V_{E/K}(n)|q^{-sn}}\geq\\
\geq\frac{\sum_{\stackrel{n\equiv1~(\mathrm{mod}~N)}{n>n_0}}\varepsilon/2\cdot|V_{E/K}(n)|q^{-sn}}{\sum_{\stackrel{n\equiv1~(\mathrm{mod}~N)}{n>n_0}}N|V_{E/K}(n)|q^{-sn}}=\frac{\varepsilon}{2N}>0.
\end{eqnarray*}
Here we used the fact that $|\#(\nu\in V_{E/K}(n)|d_\nu=1)-\delta(E/K,1,n)\cdot|V_{E/K}(n)||<c(K,E)q^{n/2}<q^n/2$ for a fixed $c(K,E)>0$ and for $n>n_0(K,E)$ depending only on $K,E$ (cf \cite{CT} Theorem 1.1 and the asymptotic formula for $|V_{E/K}(n)|$).

~\\

To prove the converse statement assume that $\delta(E/K,1,1)=0$. We will show that for all $n\in\mathbb{N}$ we have $\delta(E/K,1,n)=0$. Then there exists $C>0$ such that $|V_{E/K}(n)|\geq Cq^n/n$ and by Theorem \ref{aszimptotika} we have some $c=c(K,E,1)$ such that $\#(\nu\in V_{E/K}(n)|d_\nu=1)\geq cq^{n/2+1}/n$. So by definition
\[\delta(E/K,1)=\lim_{s\to1+0}\frac{\sum_n\#(\nu\in V_{E/K}(n)|d_\nu=1)q^{-ns}}{\sum_n|V_{E/K}(n)|q^{-ns}}\leq\lim_{s\to1+0}\frac{\sum_ncq^{1-n(s-1/2)}/n}{\sum_nCq^{-n(s-1)}/n}=0,\]
and we are done.

We shall examine in detail when is $\delta(E/K,1,1)=0$.
\[\delta(E/K,1,1)=\sum_{m|q-1}\frac{\mu(m)}{|K(E[m]):K|}=\frac{1}{|K(E[q-1]):K|}\sum_{m|q-1}|K(E[q-1]):K(E[m])|.\]
Let $H=\Gal(K(E[q-1])/K)\leq\mathrm{SL}_2(\mathbb{Z}/(q-1)\mathbb{Z})$ and for primes $\ell|q-1$ denote $H_\ell=H\cap\Ker(\pi_\ell)$, where $\pi_\ell:\mathrm{SL}_2(\mathbb{Z}/(q-1)\mathbb{Z})\to\mathrm{SL}_2(\mathbb{Z}/\ell\mathbb{Z}))$ is the modulo $\ell$ reduction. So by the inclusion exclusion principle $\delta(E/K,1,1)=0$ if and only if we have $H=\bigcup_\ell H_\ell$.

Now let $n$ be an arbitrary integer. We can refine our computation of $\delta(E/K,1,n)$ as follows:
\[\delta(E/K,1,n)=\sum_{\stackrel{m_0|q^n-1}{(m_0,q-1)=1}}\sum_{m'|q-1}\frac{\mu(m_0m')ord_{m_0m'}(q)}{|K(E[m_0m']):K|}=\sum_{\stackrel{m_0|q^n-1}{(m_0,q-1)=1}}\frac{\mu(m_0)\mathrm{ord}_{m_0}(q)}{|K(E[m_0(q-1)]):K|}S^{(m_0)},\]
where $S^{(m_0)}=\sum_{m'|q-1}\mu(m')|K(E[m_0(q-1)]):K(E[m_0m'])|$.

We claim that $S^{(m_0)}$ is 0 for all $m_0$. For this let $H^{(m_0)}=\Gal(K(E[m_0(q-1)]):K(E[m_0]))$ and for primes $\ell|q-1$ let $H^{(m_0	)}_\ell=\Gal(K(E[m_0(q-1)])/K(E[m_0\ell]))$. As before, we have \break $S^{(m_0)}=0\iff H^{(m_0)}=\bigcup_\ell H^{(m_0)}_\ell$.

Let $\sigma\in H^{(m_0)}$. We can view it as an element of $\Gal(K(E[m(q-1)])/K$, and thus \break $\overline{\sigma}=\sigma|_{K(E[q-1])}\in H$. Since $\delta(E/K,1,1)=0$ there exists $\ell|q-1$ such that $\overline{\sigma}\in H_\ell$. But this means that $\sigma$ fixes $K(E[\ell])$. Moreover by definition $\sigma$ fixes $K(E[m_0])$, thus also \break $K(E[m_0\ell])=K(E[\ell])\cdot K(E[m_0])$. Hence $\sigma\in H^{(m_0)}_\ell=\Gal(K(E[m_0(q-1)])/K(E[m_0\ell]))$ and as we can do that for any $\sigma$, we got $H^{(m_0)}=\bigcup_\ell H^{(m_0)}_\ell$.
\end{proof}

\begin{cor}
The proof shows that if $q-1$ has at most 2 prime factors $\ell_1$ and $\ell_2$, then $\delta(E/K,1)>0$ if and only if $E(K)$ has cyclic torsion subgroup.
\end{cor}

\begin{proof}
In this case $H\neq H_{\ell_1}\cup H_{\ell_2}$ - the union of two proper subgroup can not be the whole group.
\end{proof}

By the first glimpse one would expect that $\delta(E/K,1,1)=0$ if and only if the torsion subgroup of $E(K)$ is not cyclic. This is not true, in the following we construct counterexamples.

\begin{pro}
If $q-1$ has at least 3 distinct prime divisors, there exists an elliptic curve $E/K$ with cyclic torsion subgroup for which $\delta(E/K,1)=0$.
\end{pro}

\begin{proof}
If $q-1$ has at least 3 distinct prime divisors, we can construct some subgroups \break $H\leq\mathrm{SL}_2(\mathbb{Z}/(q-1)\mathbb{Z})$ such that $H=\bigcup_\ell H_\ell$.

In the case of $p=2$ we can write $q-1=q_1q_2q_3$ with $q_i>2$ and pairwise relatively prime. Let $H$ contain the central elements $\mathrm{diag}(1),\mathrm{diag}(x_1),\mathrm{diag}(x_2)$ and $\mathrm{diag}(x_3)$ of $\mathrm{SL}_2(\mathbb{Z}/(q-1)\mathbb{Z})$, where we have $x_i\equiv 1~(\mathrm{mod}~q_j)$ if $i=j$ and $x_i\equiv -1~(\mathrm{mod}~q_j)$ if $i\neq j$. Then $H\simeq (\mathbb{Z}/2\mathbb{Z})^2$ and the $H_\ell$-s are the nontrivial subgroups of $H$.

In the case of $p>2$ we can write $q-1=2^\alpha q_1q_2$ with $\alpha\geq1$, $q_i$ odd and $(q_1,q_2)=1$. There exists $r\in\mathbb{Z}/(q-1)\mathbb{Z}$ such that $r\equiv 1~(\mathrm{mod}~2^\alpha)$, $r\equiv 1~(\mathrm{mod}~q_1)$, $r\equiv-1~(\mathrm{mod}~q_2)$. Let
\[H=\left\{\left(\begin{array}{cc}1&0\\0&1\end{array}\right),\left(\begin{array}{cc}-1&0\\0&-1\end{array}\right),\left(\begin{array}{cc}r&2^{\alpha-1}q_1q_2\\0&r\end{array}\right),\left(\begin{array}{cc}-r&2^{\alpha-1}q_1q_2\\0&-r\end{array}\right)\right\}.\]
As above we have $H\simeq (\mathbb{Z}/2\mathbb{Z})^2$ and the $H_\ell$-s are the nontrivial subgroups of $H$.

For example the smallest $q$ is $q=31$, then we have
\[H=\left\{\left(\begin{array}{cc}1 & 0\\ 0 & 1\end{array}\right),\left(\begin{array}{cc}11 & 15\\ 0 & 11\end{array}\right),\left(\begin{array}{cc}19 & 15\\ 0 & 19\end{array}\right),\left(\begin{array}{cc}29 & 0\\ 0 & 29\end{array}\right)\right\}\simeq (\mathbb{Z}/2\mathbb{Z})^2\leq\mathrm{SL}_2(\mathbb{Z}/30\mathbb{Z}).\]

Now our task is to find an elliptic curve $E/K$ such that the algebraic closure of the prime field in $K$ has $q$ elements and $\Gal(K(E[q-1])/K)=H$. Let $E/\mathbb{F}_q(t)$ be a curve with $j$-invariant $t$. Then by Igusa's results (\cite{Ig}, Theorem 3) we have $N(E/\mathbb{F}_q(t))=1$. We have \break $G=\Gal(\mathbb{F}_q(t)(E[q-1])/\mathbb{F}_q(t))\simeq\mathrm{SL}_2(\mathbb{Z}/(q-1)\mathbb{Z})$, hence we can identify $G$ with the special linear group. Let $H\leq G$ the subgroup we constructed above and $K=\left(\mathbb{F}_q(t)(E[q-1])\right)^H$ and consider $E/K$. It is clear that the constant field of $K$ has size $q$ and that $\Gal(K(E[q-1])/K)=H\leq G$. Moreover the only exceptional primes are the primes dividing $q-1$, since the geometric part of the extensions $\mathbb{F}_q(t)(E[\ell])$ are disjoint.
\end{proof}

\begin{rem}
\item It does not follow from the statement that for no $\nu\in V_{E/K}$ is $E_\nu(k_\nu)$ cyclic. However we have proven that for only a few $\nu$-s this is the case.
\end{rem}

\bibliography{char-two}
\bibliographystyle{mystyle}

\end{document}